\documentclass[reqno, 11pt]{amsart}
\usepackage[letterpaper,hmargin=1in,vmargin=1in]{geometry}
\usepackage{mathrsfs, amsmath, amssymb, amsthm, verbatim, bbm} 
\usepackage[hidelinks]{hyperref}
\usepackage{comment}
\usepackage{graphicx}
\usepackage{xcolor,mathabx}
\newtheorem{lem}{Lemma}
\newtheorem{thm}{Theorem}

\newtheorem{rem}{Remark}

\DeclareMathOperator \re {Re}
\DeclareMathOperator \im {Im}

\newcommand{\Real}{\mathbb{R}}
\newcommand{\Integers}{\mathbb{Z}}

\newcommand{\Complex}{\mathbb{C}}
\newcommand{\loc}{\operatorname{loc}}
\newcommand{\CI}{\mathcal{C}^{\infty}}

\newcommand{\mcd}{\mathcal{D}}

\numberwithin{equation}{section}
\numberwithin{lem}{section}

\newcommand{\tR}{\tilde{R}}

\title[Low energy resolvent asymptotics of the Aharonov--Bohm Hamiltonian] {Low energy resolvent asymptotics of the multipole Aharonov--Bohm Hamiltonian}
\author{T. J. Christiansen, K. Datchev, and M.  Yang}
\address{Department of Mathematics, University of Missouri, Columbia, MO 65211 USA}
\email{christiansent@missouri.edu}
\address{Department of Mathematics, Purdue University, West Lafayette, IN 47907 USA}
\email{kdatchev@purdue.edu}
\address{Department of Operations Research \& Financial Engineering, Princeton University, Princeton, NJ 08544 USA}
\email{yangmx@princeton.edu}
\begin{document}
\begin{abstract}
We compute low energy asymptotics for the resolvent of the Aharonov--Bohm Hamiltonian with multiple poles for both integer and non-integer total fluxes. For integral total flux we 
reduce to prior results in black-box scattering 
while for non-integral total flux we
build on the corresponding techniques using an 
appropriately chosen model resolvent. The resolvent expansion can be used to obtain  long-time wave asymptotics for the Aharonov--Bohm Hamiltonian with multiple poles. 
An interesting phenomenon is that if the 
total flux is an integer then the 
scattering resembles even-dimensional Euclidean scattering,
while if it is half an odd integer then it resembles odd-dimensional Euclidean scattering. 
The behavior for 
 other values of total flux thus provides an `interpolation' between these. 
\end{abstract}
\maketitle

\section{Introduction}

We prove resolvent expansions near zero energy for the Aharonov--Bohm Hamiltonian \cite{ab} with multiple poles on $\mathbb R^2$. Let
\begin{gather*}
     P = (-i \vec \nabla - \vec A)^2, \qquad \vec A = \sum_{k=1}^n \alpha_k \vec A_0(x-x_k,y-y_k), \quad
     \vec A_0(x,y) = \frac{(-y,x)}{x^2+y^2} = \vec \nabla \arg(x+iy),
\end{gather*}
where $\alpha_k, x_k, y_k\in \Real$. Let $s_k = (x_k,y_k)$ and $S = \{s_1,\dots,s_n\}$ be the poles of the vector potential $\Vec{A}$. We equip $P$ with its Friedrichs domain $\mathcal D$, and we assume for convenience that $s_1$ is the origin. 

The low energy resolvent asymptotics of $P$ are governed by the value of the \textit{total flux}, defined by
\[
 \beta= \alpha_1+ \cdots + \alpha_n.
\]
Before stating our main results, we state an application to wave asymptotics from \cite{cdy}.

\subsection{Wave asymptotics}
Consider the solution $u=U(t)f_1:=\frac{\sin{t\sqrt{P}}}{\sqrt{P}}f_1$ of the wave equation 
\begin{equation}
\label{eq:waveIntro}
    \begin{cases}
    (D_t^2-P) u=0,\\
    u\rvert_{t=0}=0,\ \partial_tu\rvert_{t=0}=f_1\in\CI_c(\Real^2\setminus S).
    \end{cases}
\end{equation}
As usual, the more general problem $(D_t^2-P) u=f$, $ u\rvert_{t=0}=f_0$, $\partial_tu\rvert_{t=0}=f_1$, can then be treated by writing $u(t) = U'(t)f_0 + U(t)f_1 + \int_0^t U(t-s)f(s)ds$, but for simplicity we do not pursue this here. 
Denote functions locally in the domain $\mathcal{D}$ of $P$ by $\mathcal{D}_{\loc}.$
\begin{thm} \label{t:wintro}Let $\chi \in C_c^\infty(\Real^2)$, and $f_1,\;u$ be as in \eqref{eq:waveIntro}.  
Suppose no three elements of $S$ are colinear.  Then, as $t\rightarrow \infty$,
 \begin{enumerate}
\item \label{odd} if $\beta$ is  half an odd integer, then there is a constant $c>0$ so that $\| \chi u(t)\|_{L^2} = O(e^{-c t})$.
\item if $2\beta\not \in \Integers$, set $\mu_m= \min(\beta -\lfloor \beta \rfloor, 1+\lfloor \beta \rfloor -\beta)$, $\mu_M= \max(\beta -\lfloor \beta \rfloor, 1+\lfloor \beta \rfloor -\beta)$.   Then there is a  function $\tilde{u}\in \mathcal{D}_{\loc}$ such that  
\begin{equation*}
\| \chi (u(t) -\tilde{u}t^{-1-2\mu_m})\|_{L^2} = O(t^{-1-4\mu_m})+ O(t^{-1-2\mu_M}).
\end{equation*}
\item \label{even} if $\beta \in \Integers$, then there is a  function $\tilde{u}\in \mathcal{D}_{\loc}$ such that 
$$ 
\| \chi (u(t) -\tilde{u}t^{-1}(\log t)^{-2})\|_{L^2} = O(t^{-1}(\log t)^{-3}).$$
\end{enumerate}
\end{thm}
The exponential decay rate of the error in (\ref{odd}) is typical of the error we see in non-trapping odd-dimensional Euclidean scattering, while the decay rate of (\ref{even}) is typical of 
the even-dimensional Euclidean case.  This is a consequence of the structure of the resolvent expansion at $0$.

By Theorem 1.1 of \cite{my}, for every $\lambda_0>0$ there is $\varepsilon > 0$ such that  $\|\chi R(\lambda) \chi\|_{L^2} \le C |\lambda|^{-1}$ when $|\re \lambda| \ge \lambda_0$ and $|\im \lambda | \le \varepsilon$. Combining this with the low energy expansions of Theorems \ref{t:resint} and \ref{t:resnonint} below yields Theorem \ref{t:wintro}. See \cite{cdy} for details, including the further terms in the expansion in the second and third cases.

In the setting of Theorem \ref{t:wintro}, the long-time wave asymptotics are determined by the form of the low-energy expansion of the resolvent, which depends on the total flux $\beta$. This demonstrates the importance of the results in the next subsection.

\subsection{Low energy resolvent expansions}
Set $R(\lambda)=(P-\lambda^2)^{-1}:L^2(\Real^2)\rightarrow \mathcal{D}$ for $\im \lambda >0$.  It is shown in \cite[Section 3]{my} that as an operator from 
$L^2_c(\Real^2)$ to $\mcd_{\loc}$, this resolvent $R(\lambda)$ has a meromorphic continuation to $\Lambda$, the logarithmic cover of 
$\Complex \setminus \{0\}$.  Here we study the behavior of this resolvent near $\lambda=0$.

\subsubsection{Resolvent for integer total flux} 
Suppose first the total flux $\beta\in\mathbb Z$. In this case, as described in Section \ref{sec:int}, we can conjugate $P$ to an operator $\tilde{P}$ which is a compactly supported perturbation of the Laplacian on $\Real^2$.     
We will prove in Section \ref{s:if} that $\tilde{P}$ has no zero resonances or eigenvalues. 
As conjugation  does not change the form of the asymptotic expansion, applying  \cite[Theorem~2]{cd2} then yields the following resolvent expansion:

\begin{thm}\label{t:resint} If $\beta \in \mathbb Z$, then there are operators $B_{2j,k} \colon L^2_{c}(\mathbb R^2) \to \mathcal D_{\text{loc}}$ (i.e. mapping compactly supported functions in $L^2(\mathbb R^2)$ to functions which are locally in the domain of $P$) and a constant $a$, such that, for every $\chi \in C_0^\infty(\mathbb R^2)$, we have
\begin{equation}\label{e:resexpint}\begin{split}
\chi R(\lambda) \chi &=  \sum_{j=0}^\infty   \sum_{k=-j-1}^{j} \chi B_{2j,k} \chi  \lambda^{2j} (\log \lambda-a)^k \\
&= \chi B_{0,0} \chi + \chi  B_{0,-1} \chi (\log \lambda -a)^{-1} + \chi  B_{2,1} \chi  \lambda^2 (\log \lambda -a)  + \cdots,
\end{split}\end{equation}
with the series converging absolutely in the space of bounded operators $L^2(\mathbb R^2) \to \mathcal D$,  uniformly on sectors near zero.
\end{thm}

\noindent\textbf{Remarks.} 1. Our proof also shows that if $k \ne 0$, then $B_{2j,k}$ has  finite rank. Moreover, there is a unique function $G$ in  $\mathcal D_{\text{loc}}$ such that $PG = 0$, $(\log |(x,y)| - e^{-if}G(x,y))$ is bounded as $|(x,y)| \to \infty$, with $f$ as in \eqref{e:fdef}, and 
\begin{equation}\label{e:adef}
 B_{0,-1} = \frac 1 {2\pi} G \otimes G, \quad  a = \log 2 - \gamma - C_{\vec A} + \frac {\pi i }2,  \quad C_{\vec A} := \lim_{|(x,y)|\to\infty}\left[ \log|(x,y)| - e^{-if}G(x,y)\right].
\end{equation}
Here  $\gamma$ is Euler's constant, given by $\gamma = - \Gamma'(1) = 0.577\dots$. We are unaware of a technique for computing $G$ and $C_{\vec{A}}$ in general. However, these are analogues of the Green's function and logarithm of the logarithmic capacity which appear in the corresponding formulas for Dirichlet obstacles  \cite{cd}, and in certain symmetric situations the computations are the same.  See Section~\ref{s:examples} for  examples.

\noindent 2. We used the following definition, which will recur below: Given functions $f_n$ mapping $\Lambda$ to a Banach space $\mathcal B$, we say  $\sum_n f_n(\lambda)$ \textit{converges absolutely in $\mathcal B$, uniformly on sectors near zero} if, for any $\varphi>0$, there is $\lambda_1 >0$ such that  $\sum_n \|f_n(\lambda)\|_{\mathcal B}$ converges uniformly   on $\{\lambda \in \Lambda \colon 0<|\lambda| \le \lambda_1 \text{ and } |\arg \lambda| \le \varphi\}$.

\noindent 3. If we define the scattering matrix and scattering phase by equations (1.4) and (1.7) of \cite{cd}, then the conclusions of Theorems 2 and 3 of \cite{cd} hold for them as well. We will not discuss these in detail as they are not the main focus of the paper.

\subsubsection{Resolvent for non-integer total flux} Now suppose the total flux $\beta\notin \mathbb Z$. The operator $P$ then cannot be conjugated to a compactly supported perturbation of the free Laplacian on $\Real^2$.  We shall
see in Section \ref{sec:nonint} that we can instead conjugate $P$ to a compactly supported perturbation of an Aharonov--Bohm Hamiltonian $P_{\beta}$ with a single pole at $s_1$ and flux $\beta$.
By studying the resolvent of $P_{\beta}$ and using some resolvent identities of Vodev, we obtain
\begin{thm}\label{t:resnonint} Suppose $\beta \not \in \mathbb Z$. Set
$\mu_m= \min(\beta -\lfloor \beta \rfloor, 1+\lfloor \beta \rfloor -\beta)$, $\mu_M= \max(\beta -\lfloor \beta \rfloor, 1+\lfloor \beta \rfloor -\beta)$.
There are operators $B_{j,k}, \; B_{j,k}' \colon L^2_{c}(\mathbb R^2) \to \mathcal D_{\text{loc}}$  such that, for every $\chi \in C_0^\infty(\mathbb R^2)$, we have
\begin{equation}\label{e:resexpnonint}\begin{split}
 \chi R(\lambda) \chi&=  \sum_{j=0}^\infty  \sum_{k=0}^{\infty}  \chi B_{j,k}\chi  \lambda^{2(j+ k \mu_m)}+
\sum_{j=0}^\infty \sum_{k=1}^\infty \chi B_{j,k}' \chi  \lambda^{2(j+ k \mu_M)} ,
\end{split}\end{equation}
with the series converging absolutely in the space of bounded operators $L^2(\mathbb R^2) \to \mathcal D$, uniformly in 
$|\lambda|<\epsilon$ for some $\epsilon>0.$ Moreover, if $k>0$ then $B_{j,k}$ and $B_{j,k'}$  have finite rank.
\end{thm}

\subsection{Meromorphic continuation of the resolvent}
As mentioned earlier, \cite[Section 3]{my} shows  
 $R(\lambda)$ has a meromorphic continuation to $\Lambda$, the logarithmic cover of 
$\Complex \setminus \{0\}$.  If, however, the total flux is a non-integral
rational number, our techniques easily show that the meromorphic 
continuation descends to a smaller Riemann surface.
\begin{thm}\label{t:riem}
    Suppose the total flux $\beta$ satisfies $\beta=p/q$, where $p$ and $q$ are coprime integers, and $q \ge 2$.  Then as an operator 
$L^2_c(\Real^2)\rightarrow \mcd_{\loc}$, $R(\lambda)$ has a meromorphic continuation to $\Lambda_q$, the minimal Riemann surface on which $\lambda$ and 
$\lambda^{2/q}$ are analytic functions.
\end{thm}
Theorem \ref{t:riem} could also be deduced from \cite[Section 3]{my}. Of particular interest is the case $q=2$, i.e. $\beta \not\in \mathbb Z$ but $2\beta \in \Integers$.  Then the meromorphic continuation is to $\Complex$, the double cover of the upper half-plane, 
just as in odd-dimensional Euclidean scattering.

\subsection{Background and context}
Low frequency resolvent expansions of Schr\"odinger operators have a long history in scattering theory, explicitly since the early results of MacCamy \cite{maccamy} and implicitly even before. Because in dimension two there are several types of resonance and eigenvalue at zero, each playing a different role, this dimension is more challenging than any other -- compare the papers \cite{JeKa,je80,je84,BGD,jn} which study this problem for Schr\"{o}dinger operators with  real-valued potentials decaying sufficiently fast at infinity in dimensions respectively three, at least five, four, two (with an additional restriction) and dimension no greater than two.  A rather general abstract approach to 
resolvent expansions near $0$ can be found in \cite{MuSt}.

For magnetic Schr\"odinger operators with bounded vector potentials satisfying certain decay conditions, resolvent expansions were established in \cite{Kov2} for two dimensions and more recently in \cite{je23} for three dimensions. Note that in our setting, unlike in that of \cite{Kov2}, the resolvent has a meromorphic continuation to $\Lambda$ or $\Lambda_q$ and a complete expansion at $\lambda=0$, and we are able to prove that the resolvent is regular at $\lambda=0$ rather than having it as an assumption. 
In \cite[Section 5]{gk} the authors study the resolvent of a single pole Aharonov-Bohm type operator, as well 
as considering an additional electric potential.  
To the best of the knowledge of the authors,  our paper is the first result on the low energy resolvent asymptotics of Aharonov-Bohm operators with multiple poles.

Wave decay results, similar to Theorem \ref{t:wintro}, have been much studied for decades. The field is too wide-ranging to survey here. Let us mention the seminal work of Morawetz \cite{m61}, and the surveys in  \cite[Epilogue]{lp89}, \cite[Chapter X]{va}, \cite{dr}, \cite{tat}, \cite{dz}, \cite{vasy}, \cite{sch21}, \cite{klainerman}. Some results in settings closer to the present one include \cite{Mur,Kov}. For the Aharonov-Bohm Hamiltonian, various wave decay results have only been established for the Hamiltonian with a single pole in \cite{fffp,gk,fzz,wang23}, which has scaling and rotational symmetry.

\subsection{Plan of the paper}
In Section \ref{s:pre}, we introduce some preliminaries of the Aharonov--Bohm Hamiltonian and the construction of the unitary conjugation operator for both integer and non-integer total flux. In Section \ref{s:if}, we prove Theorem \ref{t:resint} when the total flux is an integer. In Section~\ref{sec:res-non}, we prove Theorem \ref{t:resnonint} for non-integral total flux.

\subsection*{Acknowledgements}
The authors would like to thank Luc Hillairet for helpful discussions, as well as Daniel Tataru and Maciej Zworski for proposing this project and helpful discussions. The authors appreciate the helpful comments and corrections of
the referees. TC and KD are grateful for Simons collaboration grants for mathematicians for travel support. MY is partially supported by the NSF grant DMS-1952939 and DMS-2509989.

\subsection*{List of Notation}
\begin{itemize}
    \item The set of $n$ poles: $S = \{s_i=(x_i,y_i)\in\Real^2: 1\leq i\leq n\}$;
    \item $\Gamma = \Gamma_2\cup \cdots \cup \Gamma_n$, and each $\Gamma_k$ is the segment joining $s_1$ and $s_k$;
    \item $\Omega = \mathbb{R}^2\backslash\Gamma$; 
    \item The Aharonov--Bohm Hamiltonian $P = (-i\nabla-\Vec{A})^2$ on $\mathbb{R}^2\backslash S$ with Friedrichs domain $\mathcal{D}$, where $\Vec{A}$ is given in equation \eqref{potential};
    \item The conjugated operator we work with throughout the paper: $\tilde{P}=e^{-if} P e^{if}$ on $\Omega$ with the corresponding domain $\Tilde{\mathcal{D}}$;
    \item The ``free" Hamiltonian (model operator): $P_{\beta}= (-i\nabla-\beta\Vec{A_0})^2$ on $\mathbb{R}^2\backslash \{0\}$ with Friedrichs domain $\mathcal{D}_{\beta}$; the operator $\tilde{P}$ is a compactly supported perturbation of $P_{\beta}$; 
    \item The resolvent of $P$: $R(\lambda)=(P-\lambda^2)^{-1}$ on 
    $\Real^2\setminus S$;
    \item The resolvent of the conjugated operator $\tilde P$: $\tR(\lambda)=(\Tilde{P}-\lambda^2)^{-1}$ on $\Omega$;
    \item The ``free" resolvent (model resolvent): $R_{\beta}(\lambda)=(P_{\beta}-\lambda^2)^{-1}$ on $\mathbb{R}^2\backslash \{0\}$.
\end{itemize}

\section{Preliminaries}
\label{s:pre}
This section contains basic facts about $P$ and a construction
of the conjugated operator $\tilde{P}$.
\subsection{Operators and domains}
We study the magnetic Hamiltonian 
\begin{equation}
\label{TotHam}
    P = (-i \vec \nabla - \vec A)^2,
\end{equation}
on the space $X:=\Real^2\setminus S$, where $S=\{s_i=(x_i,y_i)|1\leq i\leq n\}$ corresponds to locations of the $n$ poles of $\vec A$ and
\begin{equation}
\label{potential}
     \vec A = \sum_{k=1}^n \alpha_k \vec A_0(x-x_k,y-y_k), \quad
     \vec A_0(x,y) = \frac{(-y,x)}{x^2+y^2} = \vec \nabla \arg(x+iy)
\end{equation}
with $\alpha_i\notin\mathbb{Z}$. Note that the magnetic potential $\Vec{A}$ is singular at $s_i$ for $1\leq i\leq n$ and curl-free; therefore there is no magnetic field in $\Real^2\setminus S$. The singular magnetic potential $\Vec{A}$ is related to the famous Aharonov--Bohm effect \cite{ab}. 

Note that the operator $P$ with domain $\CI_c(X)$ admits various self-adjoint extensions, as it is a positive symmetric operator defined on $\CI_c(X)\subset L^2(\Real^2)$ with deficiency indices $(2n,2n)$. In this paper we consider the Friedrichs self-adjoint extension, which is the only self-adjoint extension of $P$ whose domain $\mathcal{D}$ is contained in the closure of the quadratic form domain:
$$\overline{\left\{u\in L^2: \langle Pu,u\rangle_{L^2} + \|u\|^2_{L^2} <\infty \right\}}.$$
As a result, the domain can be characterized by
\begin{equation}
    \label{eq:dom}
    \mathcal{D}=\left\{u\in L^2: P u\in L^2, u(x)\to 0 \text{ as } x\to S\right\};
\end{equation}
see also \eqref{eq:intk} for the resolvent kernel formula when $n=1$.
Physically, the Friedrichs extension corresponds to the poles being impenetrable so that wave functions vanish there.  For detailed discussions of self-adjoint extensions of the Aharonov--Bohm Hamiltonian, see \cite{at,ds,f24}, or   \cite{cf23} on the Aharonov--Bohm Hamiltonian with multiple poles. 

\subsection{Conjugated operators}
To use  perturbation theory to study the resolvent expansion, we need to define a unitary conjugation to transform the operator $P$ nicely, in particular, outside a compact set.

\subsubsection{Integer flux}
\label{sec:int}
We first consider $\beta\in\mathbb Z$. Let $\Omega := \mathbb R^2 \setminus \Gamma$, where $\Gamma := \Gamma_2\cup \cdots \cup \Gamma_n$, and each $\Gamma_k$ is the segment joining $(x_1,y_1)$ and $(x_k, y_k)$. Fix some $(x_0,y_0) \in \Omega$. For each $(x,y) \in \Omega$, let 
\begin{equation}\label{e:fdef}
f(x,y):= \int_{\gamma} \Vec{A} \cdot d\Vec{\gamma}.
\end{equation}
where $\gamma$ is a path in $\Omega$ from $(x_0,y_0)$ to $(x,y)$. By the definition of $\Vec{A}$ in \eqref{potential}, changing the choice of path only changes the value $f$ by an integer multiple of $2 \pi$ and hence $e^{if}$ is independent of the choice of path. The conjugated operator we use is
\begin{equation} \label{e:peif}
 \tilde P := e^{-if} P e^{if} = (-i\vec \nabla + \vec\nabla f - \vec A)^2 =  -\Delta, \qquad  \text{ on } \Omega.
\end{equation}

Now we consider how the domain $\mathcal{D}$ of $P$ transforms, under the unitary conjugation, to the domain $\tilde{\mathcal{D}} $ of $\Delta$ on $\Omega$ using the definition \eqref{e:peif}. We define 
\begin{equation}
    \label{eq:conj_dom}
    \tilde{\mathcal{D}} := \{ v = e^{-if}\cdot (u\rvert_{\Omega}),\ u\in \mathcal{D} \}.
\end{equation}
Then $v\in \tilde{\mathcal{D}}$ extends continuously to $\partial \Omega$, provided we distinguish the two sides of each segment $\Gamma_j$. More specifically, $v\rvert_S = 0$, and if $z = (x,y) \in \Gamma \setminus S$, then we claim
\begin{equation}
  \label{eq:matching}
 v_+ = e^{-2\pi i \tilde \alpha} v_-, \quad \partial_\nu v_+ = - e^{-2\pi i \tilde \alpha} \partial_\nu v_-,\quad 2 \pi \tilde \alpha = f_+(z) - f_-(z)
\end{equation}
where $v_\pm(z) = \lim_{\varepsilon \to 0^+} v(e^{\pm i\varepsilon} z)$, and we are using the convention that the normal derivatives $\partial_{\nu}v_+$ and $\partial_\nu v_-$ are pointing in opposite directions. This is because for $v=e^{-if}u$, we have $v_-={e^{-if_-(z)}u}$ and $v_+={e^{-if_+(z)}u}$. Therefore $$v_- = e^{-i(f_-(z)-f_+(z))}v_+ = e^{2\pi i \tilde{\alpha}} v_+,$$ since for $\Vec{A}$ defined in \eqref{potential}
$$f_+(z)-f_-(z)=\oint_{\gamma_z}\Vec{A}\cdot d\gamma = 2\pi \tilde{\alpha}\; \mod 2\pi,$$
where $\gamma_z$ is a simple closed curve in $\Omega$
from $\lim_{\varepsilon \to 0^+}e^{-i\varepsilon} z$ to $\lim_{\varepsilon \to 0^+} e^{i\varepsilon} z$
and having positive orientation. Thus $\tilde \alpha(z) = \alpha_j$ when $z \in \Gamma_j$ and there is no $s_k$ such that $s_j$ lies on the segment between $s_1$ and $s_k$, and otherwise $\tilde \alpha(z)$ is a sum of all the $\alpha_j$ such that $z$ lies on the segment between $s_1$ and $s_j$; see Figure~\ref{f:graph} for an example.

\begin{figure}[ht]
 \includegraphics[width=6cm]{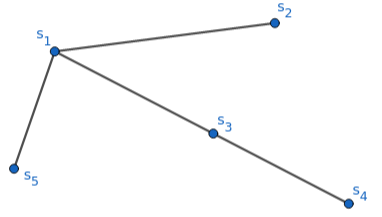}
\caption{In this example, $\tilde \alpha = \alpha_j$ on $\Gamma_j$ for $j = 2$ and $j=5$, $\tilde \alpha = \alpha_3 + \alpha_4$ on $\Gamma_3$, and $\tilde \alpha = \alpha_4$ on $\Gamma_4 \setminus \Gamma_3$.}\label{f:graph}
\end{figure}

\subsubsection{Non-integer flux}
\label{sec:nonint}
Now we consider the total flux $\beta\notin \mathbb Z$. Recall that $\Omega = \mathbb R^2 \setminus \Gamma$, with $\Gamma = \Gamma_2\cup \cdots \cup \Gamma_n$, and each $\Gamma_k$ is the segment joining $(x_1,y_1)$ and $(x_k, y_k)$. We define a new phase function in the conjugation. Fix some $(x_0,y_0) \in \Omega$. For each $(x,y) \in \Omega$, let 
$$f(x,y):= \int_{\gamma} (\vec A -  \beta \vec A_0) \cdot d\Vec{\gamma}$$
where $\gamma$ is a path in $\Omega$ from $(x_0,y_0)$ to $(x,y)$. By equation \eqref{potential}, the functions $f$ and $e^{if}$ are independent of the choice of path. In particular, we have
\begin{equation}
\label{eq:U-equiv}
    \tilde{P}:=e^{-if} P e^{if} = P_\beta,
\end{equation}
when restricted to $\Omega$, where $P_{\beta} = (-i\nabla-\beta\Vec{A_0})^2$ on $\mathbb{R}^2\backslash \{0\}$ is the Aharonov--Bohm Hamiltonian with one pole and flux $\beta$. Hence, the operator $\tilde P$ is a compactly supported perturbation of $P_{\beta}$. Note that as in the case of integral total flux, under the unitary transform the domain $\Tilde{\mathcal{D}}$ of $\tilde{P}$ defined via \eqref{eq:conj_dom} using the new phase function satisfies the matching condition \eqref{eq:matching}.

\section{Resolvent expansion for integer total flux}\label{s:if} 

\subsection{Proof of  Theorem \ref{t:resint}} The proof of Theorem \ref{t:resint} is short because it follows from \cite[Theorem 2]{cd2}, where the resolvent expansion of a compact black-box perturbation of Laplacian is given under a certain non-resonance condition. It suffices to verify that this condition holds.

More specifically, $\tilde P$ is a black-box perturbation of $-\Delta$ on $\mathbb R^2$ in the sense of \cite{sz} (see also Chapter 4 of \cite{dz}): it is self-adjoint, nonnegative, and has $\chi \tilde R(\lambda)$ compact on $L^2(\Omega)$ for $\chi \in C_c^\infty(\mathbb R^2)$ and $\im \lambda >0$ by the following lemma.

\begin{lem}\label{l:quadform} For all $u$ and $v$ in $\tilde{\mathcal D}$, we have
\begin{equation}\label{e:quadform}
 \int_{\Omega} \nabla u \cdot \nabla \bar v  = \int_\Omega (-\Delta u) \bar v.
 \end{equation}
 In particular, $\Tilde{\mathcal D} \subset H^1(\Omega)$.
\end{lem}
\begin{proof}
By Green's identity, since $\Omega$ is a union of sectors, we have
\begin{equation*}
  \int_{\Omega} \nabla u \cdot \nabla \bar v  = \int_{\Omega} (-\Delta u) \bar v + \int_{\partial \Omega}(\partial_{\nu}u ) \bar v,
\end{equation*}
where, as in \eqref{eq:matching}, we distinguish the two sides of each segment $\Gamma_j$ in $\partial \Omega$. But $\int_{\partial \Omega}(\partial_{\nu}u ) \bar v = 0$ because by \eqref{eq:matching} the function $(\partial_{\nu}u_+ ) \bar v_+ + (\partial_{\nu}u_- ) \bar v_-$ vanishes identically on $\Gamma \setminus S$.
\end{proof}

By the Remark following Theorem 2 of \cite{cd2}, to invoke that result it is now enough to check that $\Tilde{P}u=0$ has no bounded solutions in $\Tilde{\mathcal D}_{\textrm{loc}}$. For that, recall that  if $u$ is harmonic and bounded on $\{ x \in \mathbb R^2: \; |x|>\rho\}$, then there are constants $c_0$, $c_{j,c}$, $c_{j,s}$, such that
\begin{equation}\label{eq:harmexp}
 u(r\cos \theta, r\sin \theta)=c_0+ \sum_{j=1}^\infty (c_{j,c}\cos j\theta + c_{j,s} \sin j\theta)r^{-j}, \qquad \text{for }r>\rho.
\end{equation}

\begin{lem}\label{l:noresat0}
If $\tilde{P}u=0$ and $u \in \Tilde{\mathcal D}_{\textrm{loc}}$ is bounded, then $u$ is identically zero.
\end{lem}

\begin{proof}
Let $D_\rho = \{x \in \mathbb R^2 \colon |x|<\rho\}$. By \eqref{e:quadform} and \eqref{eq:harmexp},
\[
 \int_{\Omega \cap D_\rho} |\nabla u|^2 = \int_{\partial D_\rho} u\, \partial_r \bar u \, dS= O(\rho^{-1}), \qquad \text{ as } \rho \to \infty,
\]
which implies that $\nabla u$ is identically $0$. Since $u \to 0$ at points of $S$ by \eqref{eq:dom}, it follows that $u$ is identically zero.
\end{proof}

Lemma \ref{l:noresat0} shows that, in the notation of \cite[(1.3)]{cd2}, 
$\mathcal{G}_0=\{0\}$.  Since $\mathcal{G}_{-1}\subset \mathcal{G}_0$
it follows from \cite[Corollary 4.3 and Theorem 2]{cd2} 
that $\tR(\lambda)$
has an asymptotic expansion near $0$ of the form of \eqref{e:resexpint}.  Since $R(\lambda)= e^{if}\tR(\lambda)e^{-if}$, Theorem \ref{t:resint} follows.

\subsection{Examples}\label{s:examples} The following examples illustrate the function $G$ and the constant $C_{\Vec{A}}$ in \eqref{e:adef}. In these examples we use the standard identification of $\mathbb R^2$ with $\mathbb C$. Let $n\ge 3$ and $\rho>0$ be given, and define an Aharonov--Bohm potential and the cut $\Gamma$ as in Figure \ref{f:examples}.

\begin{figure}[ht]
 \includegraphics[width=15cm]{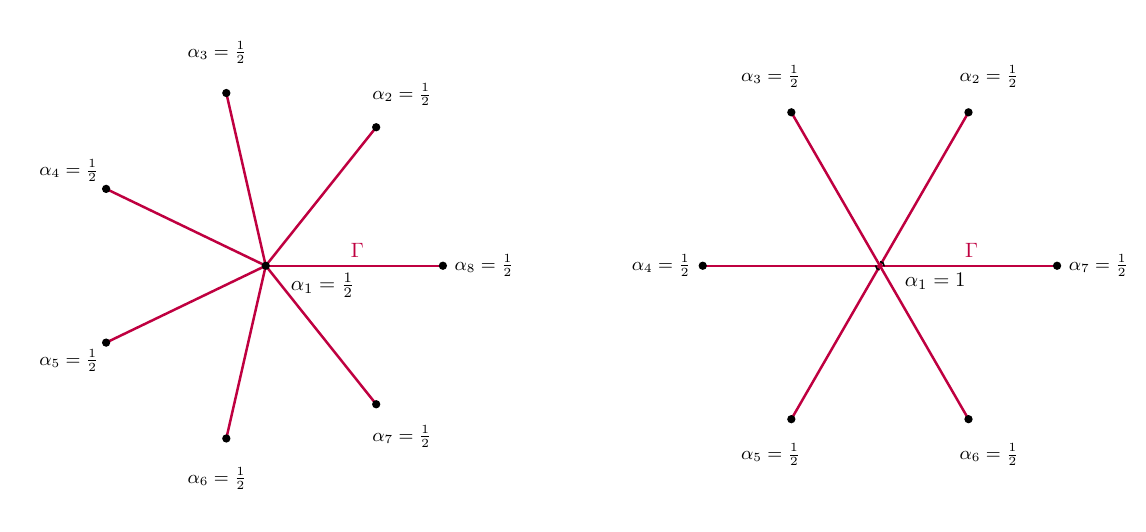}
\caption{Examples for which we can compute $G$ and $C_{\vec A}$.}
\label{f:examples}
\end{figure}

Explicitly:
\begin{itemize}
\item There are poles $s_k=\rho e^{2\pi i(k-1)/(n-1)}$, $k=2,\dots ,n$, each having flux $\alpha_k=1/2$.  
\item If $n$ is even, there is a pole at the origin with flux $\alpha_1=1/2$. If $n$ is odd, there is a (phantom) pole at the origin with flux $\alpha_1=1$.
\end{itemize}

Because we have taken the pole at the origin to be $s_1$, $\Gamma$  as defined in Section \ref{sec:int} is given by 
$$\Gamma=\bigcup_{k=1}^{n-1}e^{2\pi i(k-1)/(n-1)}[0,\rho].$$  
Because of the way the we have chosen the fluxes,  \eqref{eq:matching} simplifies to
\begin{equation}\label{eq:specialG}
v_+=-v_-,\;\; \partial_\nu v_+=\partial_\nu v_-.
\end{equation}

Let $\tilde{G}$ be the function in $\tilde{\mathcal{D}}_{\loc}$ which satisfies $\tilde{P}\tilde{G}=0$, $\tilde{G}(z)= \log|z|+O(1)$; uniqueness of such a function follows from Lemma \ref{l:noresat0}.  Note
that $\tilde{G}=e^{-if}G$, where $G$ is the function defined in \eqref{e:adef}.  

By the symmetry of $\Gamma$, uniqueness of $\tilde G$, and  \eqref{eq:specialG}, we have
\begin{equation}\label{e:Gsym}
 \tilde G(\bar z) = \tilde G(z) = \tilde G(e^{2\pi i(k-1)/(n-1)} z).
\end{equation}
Combining the first of \eqref{e:Gsym} with the first of \eqref{eq:specialG} shows that $\tilde G(z) \to 0$ as $z \to  z_0 \in (0,\rho)$. Using also the second of \eqref{e:Gsym} shows that $\tilde G(z) \to 0$ as $z \to  z_0 \in \Gamma$.  Since $\tilde G$ is harmonic on $\mathbb C \setminus \Gamma$,  it follows by Theorem 5.2.1 of  \cite{ransford} that $e^{C_{\vec{A}}}$ is the logarithmic capacity of $\Gamma$, i.e.  
 \[
 C_{\vec A} = \log (2^{-2/(n-1)}\rho).
 \]
More explicitly, when  $n=3$,   using the Joukowsky transform $z = w + \rho^2/4w$  yields  
\begin{equation}\label{e:gexplicit}
\tilde G(z) = \log \Bigg| \frac z \rho \Bigg(1 + \sqrt{1 - \frac {\rho^2}{z^2}}\Bigg)\Bigg|.
\end{equation}
The formula \eqref{e:gexplicit} can be adapted to other values of $n$  by taking the inverse image under a  complex polynomial as in \cite[Theorem 5.2.5]{ransford}. 

\section{Resolvent expansion for non-integer total flux}
\label{sec:res-non}

Now we consider the resolvent when $\beta\notin\mathbb Z$. For convenience, in this section
we assume that $0<|\beta|\le 1/2$ since the general case is unitarily equivalent to this one. 

\subsection{Model resolvent}
\label{sec:free}
First, we have the following lemma on the asymptotics of the ``model" resolvent $R_{\beta}(\lambda)=(P_\beta-\lambda^2)^{-1}$.  We denote
by $\mathcal{D}_\beta$ the domain of $P_\beta, $ and by 
$\mathcal{D}_{\beta;\loc} $ elements of $L^2_{\loc}$ which are locally in 
$\mathcal{D}_\beta$.  A related result for the 
resolvent for the Laplacian on cones can be found in 
\cite[Section 7.2]{MuSt}.

\begin{lem}\label{l:modelr}  There are operators
$A_0(\lambda),\; \tilde{A}_{\pm}(\lambda):L^2_c(\Real^2)\rightarrow \mathcal{D}_{\beta; \loc}$, depending analytically on 
$\lambda^2$
such that the model resolvent
\begin{equation}\label{eq:3term}
R_\beta(\lambda) = A_0(\lambda)+\lambda^{2|\beta|} \tilde{A}_+(\lambda) 
+\lambda^{2(1-|\beta|)}\tilde{A}_{-}(\lambda).
\end{equation}
 In particular, $R_\beta$
 has an expansion near zero of the form
\begin{equation}\label{e:modelresbeta}
 R_\beta(\lambda) = \sum_{j\geq 0} \Big(R_{2j,0} + R_{2j,-}  \lambda^{2|\beta|} + R_{2j,+} \lambda^{2(1-|\beta|)} \Big)\lambda^{2j}
\end{equation}
where the  $R_{2j,0}, $ $R_{2j,\pm}$ are operators mapping compactly supported $L^2$ functions to functions locally in $\mathcal{D}_{\beta}$ 
 and the series converges absolutely and uniformly near $\lambda = 0$.
Moreover, $R_{2j,\pm}$ are finite rank for each $j$.
\end{lem}

\begin{proof} Define 
$$\nu_l=| l-\beta|.$$
We use results of equation (17) of \cite{my}\footnote{There is a minus sign missing in the resolvent formula (17) of \cite{my}.}.  If 
$$f(r,\theta)=\sum_{-\infty}^\infty f_l(r)e^{i l \theta},$$
then 
\begin{equation}\label{eq:sv}
(R_\beta(\lambda)f)(r,\theta)=\sum_{-\infty}^\infty (R_\beta(\lambda;l)f_l)(r)e^{i l \theta}\end{equation}
  and the problem reduces to expanding the kernels
\begin{equation}\label{eq:intk}
 R_\beta(r,\tilde r;\lambda;l) = \frac{i\pi}{2} J_{\nu_l}(\lambda r_<) H_{\nu_l}^{(1)}(\lambda r_>) \tilde r, \qquad r_<=\min(r,\tilde r), \ r_> = \max(r, \tilde r),
\end{equation}
for bounded $r$, $\tilde r$ and  $\lambda$ and obtaining uniform bounds.

For that we use the series
\[
 J_{\nu}(z) = (z/2)^\nu \sum_{k=0}^\infty \frac{(-z^2/4)^k}{k!\Gamma(\nu+k+1)} 
\]
and
\[
H^{(1)}_\nu(z) = (1+ i\cot(\nu \pi))J_\nu(z) - i \csc(\nu \pi) J_{-\nu}(z).
\]
Note that this means $\lambda^{-2\nu_l}J_{\nu_l}(\lambda r_<)J_{\nu_l}(\lambda {r}_>)$ and $J_{\nu_l}(\lambda r_<)J_{-\nu_l}(\lambda r_>)$ are entire, even functions of $\lambda$.  This observation is the origin
of the decomposition \eqref{eq:3term}.

As $\nu_l\rightarrow \infty$, by equation (1) of Section 3.13 of \cite{watson} we have $J_{\nu_l}(z) = \frac{(z/2)^{\nu_l}}{\Gamma(\nu_l  +1)}(1 + O(|z|^2/\nu_l))$, and combining this with Stirling's approximation  $\Gamma(\nu+1)\sim \sqrt{2\pi \nu}(\nu/e)^\nu$ as in Section 8.1 of \cite{watson}, shows that for $\lambda r,\; \lambda\tilde r$ varying in a compact set we have
\begin{align}\label{eq:JJ}
\lambda^{-2\nu_l}(1+ i\cot(\nu_l \pi))J_{\nu_l}(\lambda r_> )J_{\nu_l} (\lambda r_<)& =\lambda^{-2\nu_l}(1+ i\cot(\nu_l \pi))J_{\nu_l}(\lambda r )J_{\nu_l} (\lambda \tilde r)\nonumber \\
& \sim (1+ i\cot(\nu_l \pi))\frac{1}{2\pi \nu_l}\left(\frac{e^2 r \tilde r}{4 \nu^2_l}\right)^{\nu_l}.
\end{align}
We note that $\cot(\nu_l \pi)= \cot(\nu_{l'}\pi)$ if $l-\beta$ and $l'-\beta$ have the same sign. Thus, given $M$, $\lambda_0$, there is a 
$C>0$ such that for all $l$ we have
\begin{equation}\label{eq:JJ2}
| \lambda^{-2\nu_l}(1+ i\cot(\nu_l \pi))J_{\nu_l}(\lambda r_<)J_{\nu_l} (\lambda r_>)|\leq C \frac {(eM)^{2\nu_l}}{\nu_l^{2\nu_l+1}}\;\; \text{if $0\leq r,\;\tilde r \leq M$, \; $|\lambda|<|\lambda_0|$}.
\end{equation}

Similarly, as $\nu_l \rightarrow \infty$, for $\lambda r,\; \lambda\tilde r$ varying in a compact set, using $\sin(\pi \nu)\Gamma(1+\nu)\Gamma(1-\nu) = \pi \nu$ yields
\begin{align}
|\csc(\nu_l \pi)J_{\nu_l}(\lambda r_<)J_{-\nu_l}(\lambda r_>)| 
& \sim \frac 1 { \pi \nu_l}\left( \frac{r_<}{r_>}\right)^{\nu_l}.
\end{align}
Again, 
this implies that given $M$, $\lambda_0$, there is a 
$C>0$ such that for all $l$
\begin{equation}\label{eq:JJn}
\left|   \csc(\nu_l \pi)  J_{\nu_l}(\lambda r_<)J_{-\nu_l} (\lambda r_>)\right |\leq \frac{C}{\nu_l}\;\; \text{if $0\leq r,\;\tilde r \leq M$, \; $|\lambda|<|\lambda_0|$}.
\end{equation}

Let $\mu_+ = 1-\beta$ if $\beta \in (0,1/2]$ and $\mu_+ = -\beta$ if $\beta \in [-1/2,0)$. Let $\mu_- = 1-\mu_+$. Define operators $A_{0,l}$, $A_{\pm,l}$,  by
\begin{equation}\begin{split}\label{eq:As}
 A_{0,l}(\lambda)f(r,\theta)& =  \frac \pi{2}  \int_0^\infty   \csc(\nu_l \pi)   J_{\nu_l}(\lambda r_<)   J_{-\nu_l}(\lambda r_>)f_l(\tilde r)\tilde r d\tilde r e^{il\theta},\\
A_{+,l}(\lambda)f(r,\theta)& =\frac {i\pi}{2} \lambda^{-2 \mu_+} \int_0^\infty (1+i \cot(\nu_l \pi))J_{\nu_l}(\lambda r )J_{\nu_l}(\lambda \tilde r)f_l(\tilde r)\tilde r d\tilde r e^{il\theta}, \quad \text{for } l - \beta>0,\\
A_{-,l}(\lambda)f(r,\theta)& =\frac {i\pi}{2} \lambda^{-2\mu_-} \int_0^\infty (1+i \cot(\nu_l \pi))J_{\nu_l}(\lambda r )J_{\nu_l}(\lambda \tilde r)f_l(\tilde r)\tilde r d\tilde r e^{il\theta}, \quad \text{for } l - \beta < 0.
\end{split}\end{equation}

We claim that
\begin{equation}\label{e:resthree}
 R_\beta (\lambda)= \lambda^{2 \mu_+}A_+(\lambda)+\lambda^{2\mu_-} A_-(\lambda)+ A_0(\lambda),
\end{equation}
where
\begin{equation}\label{eq:Asd}
 A_{\pm}(\lambda)= \lim_{L\rightarrow \infty}\sum_{{0<\pm (l-\beta)<L}}A_{\pm,l}(\lambda), \quad \text{and } A_0(\lambda)=\lim_{L\rightarrow \infty}\sum_{-L<l<L}A_{0,l}(\lambda),
\end{equation}
with the series converging  uniformly on compact subsets in $\lambda$
as an operator $L^2_c\rightarrow L^2_{\loc}$. To check this claim, note that, for any radial $\chi \in C_c^\infty (\Real^2)$,
$$\Big\| \sum_{{0<\pm (l-\beta)<L}} \chi A_{\pm, l}\chi \Big\|= \max_{{0<\pm (l-\beta)<L}} \| \chi A_{\pm,l}\chi\|.$$
 Then the estimate \eqref{eq:JJ2} shows that for any  $M\in \Real_+$ there
is a $C'$ independent of $L$ such that, for $|\lambda|<M$, we have 
$$\Big\| \sum_{{0<\pm (l-\beta)<L}}\chi A_{\pm,l} \chi f\Big\|<C'\| f\|.$$
Moreover, using \eqref{eq:JJ2} again shows the  sequence $\{  \sum_{{0<\pm (l-\beta)<L}}\chi A_{\pm,l} (\lambda) \chi\}$ is Cauchy, uniformly in $\lambda$ with $|\lambda|<M$. Since
$M$ is arbitrary, this completes the proof of the claim for the first part of \eqref{eq:Asd}; the second part follows similarly using equation \eqref{eq:JJn}.

Since each $A_{\pm,l}$ is analytic in $\lambda^2$, $A_\pm(\lambda)$ is  a uniform (on compact sets) limit of analytic operator-valued functions $L^2_c(\mathbb R^2) \to L^2_{\loc}(\mathbb R^2)$, and hence is analytic
and even
 on $\Complex$. We similarly show that $A_0(\lambda)$  is analytic and even, using  \eqref{eq:JJn}.  
Setting $\tilde{A}_{\pm}=A_\pm$ if $\beta >0$ and $\tilde{A}_{\pm}=A_\mp$ if $\beta<0$ proves \eqref{eq:3term}.
 Inserting the power series expansions of $A_\pm(\lambda)$ and $A_0(\lambda)$ into \eqref{e:resthree} gives \eqref{e:modelresbeta}.
\end{proof}

We shall later need the following lemma, which computes the integral kernel of $R_\beta(0)$.
\begin{lem}\label{l:modatinfty}
Let $R_{00}$ be as defined by the equation \eqref{e:modelresbeta} in Lemma \ref{l:modelr}.  Then for any $f \in L^2_c(\Real^2)$,
$$(R_{00}f)(r,\theta)= \frac{1}{2} \sum_{l=-\infty}^\infty\frac{1}{\nu_l} \int_0^\infty \left(\frac{r_<}{r_>}\right)^{\nu_l}f_l(\tilde{r})\tilde{r}d\tilde{r}e^{il \theta}$$
in $\mathcal{D}_{\beta,\loc}$, which are functions that locally are in the domain $\mathcal{D}_{\beta}$.
\end{lem}
\begin{proof}
We use $A_0(\lambda)$ from \eqref{eq:Asd} and $A_{0,l}(\lambda)$ from \eqref{eq:As}.
From the proof of Lemma \ref{l:modelr} we see  $R_{00}=\lim_{\lambda \rightarrow 0} A_0(\lambda)$.  The proof of Lemma \ref{l:modelr} also shows that for $\chi \in C_c^\infty(\Real^2)$, 
the sequence $\{ \sum_{|l|<L}\chi  A_{0,l}(\lambda)\chi\}$ converges uniformly for $\lambda$ in a compact  subset of $\Complex$.  Hence 
$\chi R_{00}\chi = \sum_{l=-\infty}^\infty A_{0,l}(0)$.  
Now we note that for a fixed $l$,  $r, \tilde{r}$ in a bounded set and  $\lambda$ near $0$ from the asymptotic expansions of $J_{\nu_l}$, $J_{-\nu_l}$, 
$$J_{\nu_l}(\lambda r_<) J_{-\nu_l}(\lambda r_>)= \frac{1}{\Gamma(\nu_l+1)\Gamma(-\nu_l+1)}\left( \frac{r_<}{r_>}\right)^{\nu_l}+O(\lambda^2)= \frac{\sin(\pi \nu_l)}{\pi \nu_l}\left( \frac{r_<}{r_>}\right)^{\nu_l}+O(\lambda^2), $$
so that 
$$(A_{0,l}(0)f)(r,\theta)= \frac{1}{2\nu_l} \int_0^\infty \left(\frac{r_<}{r_>}\right)^{\nu_l}f_l(\tilde{r})\tilde{r}d\tilde{r}e^{il \theta}$$
for $r$ in any fixed compact set, proving the lemma.
\end{proof}

Notice that \eqref{eq:3term} shows that if $\beta =p/q\in \mathbb{Q}$ with $p,q$ being coprime integers,
then $R_\beta(\lambda)$ continues meromorphically to $\Lambda_q$,  the minimal Riemann surface on which 
$\lambda$ and $\lambda^{2/q}$ are
analytic.   In particular, if $\beta=1/2$, then $\Lambda_q = \mathbb C$.

\subsection{Vodev's identity and meromorphic continuation}
\label{s:vod}
By equation \eqref{eq:U-equiv}, away from $\Gamma$ the conjugated operator $\tilde P$ agrees with the model operator $P_{\beta}$. Consequently, arguing as in Section 2.5 of \cite{cd2} yields the resolvent identity
\begin{gather}
\nonumber \tR(\lambda)-\tR(z)= (\lambda^2-z^2)\tR(\lambda) \chi_1(2-\chi_1)\tR(z) + \{ 1-\chi_1+\tR(\lambda)[P_{\beta},\chi_1]\}(R_{\beta}(\lambda)-R_{\beta}(z))K_1, \\
 K_1 = 1 - \chi_1 -[P_{\beta},\chi_1] \tR(z), \label{e:vodevido}
\end{gather}
for any $\chi_1 \in C_c^\infty(\mathbb R^2)$ which is $1$ near $\Gamma$, and for any $\lambda$ and $z$ in the upper half plane. Bringing the $\tR(\lambda)$ terms to the left, the remaining terms to the right, and factoring, yields
\begin{equation}\label{e:vodevid}
 \tR(\lambda)  (I - K(\lambda)) = F(\lambda),
\end{equation}
where
\begin{align}\label{eq:F}
 K(\lambda) &= (\lambda^2 - z^2)\chi_1(2-\chi_1)\tR (z)  + [P_{\beta}, \chi_1](R_{\beta}(\lambda) - R_{\beta}(z))K_1, \nonumber\\
 F(\lambda) &= \tR(z) + (1-\chi_1)(R_{\beta}(\lambda) - R_{\beta}(z))K_1 .
\end{align}
Here and below we shorten formulas by using notation which displays $\lambda$-dependence but not $z$-dependence for operators other than resolvents.  The identities \eqref{e:vodevido} and \eqref{e:vodevid} are versions of Vodev's resolvent identity from \cite{vodev}.

For any $\chi \in C_0^\infty(\mathbb R^2),$ the resolvent $\tR(\lambda)$ continues meromorphically to $\Lambda$, the logarithmic cover of $\Complex \setminus \{0\}$. This has been proved for the resolvent $R(\lambda)$ in \cite{my} and the meromorphic continuation of $\tR(\lambda)$ follows from the unitary conjugation. Alternatively, it can be deduced from \eqref{e:vodevid} using the Analytic Fredholm Theorem as at the end of Section 2 of \cite{cd2}. Thus \eqref{e:vodevido} and \eqref{e:vodevid} continue to hold for any $z$ and $\lambda$ in $\Lambda$, with $K(\lambda)$ and $K_1$ mapping $L^2_c(\Omega)$ to $L^2_c(\Omega)$, and $\tR(\lambda)$ and $F(\lambda)$ mapping $L^2_c(\Omega)\rightarrow \Tilde{\mathcal{D}}_{\loc}$. 

To show Theorem \ref{t:riem}, we note that using \eqref{e:vodevid} we can see that the minimal  Riemann surface to which $\tR(\lambda)$ (and hence $R(\lambda)$) continues is the same as the Riemann surface to which $R_\beta(\lambda)$ 
continues.  Using equations \eqref{eq:sv} and \eqref{eq:intk}, for nonintegral  $\beta$ this is the Riemann surface to which the set $\{ (J_{\nu_l}(\lambda))^2,\; l\in \Integers\}$ continues.  When $\beta$ is rational but
$\beta \not \in \Integers$, this Riemann surface is a finite cover of the complex plane. In particular, if $2\beta\in\mathbb Z$, then the continuation is to the complex plane
(the double cover of the upper half plane).  This is the same Riemann surface to which, for example,  for $V\in L^\infty_c(\Real^d)$ and $\Delta=\sum_{j=1}^d \partial_{x_j}^2$,
$(-\Delta+V-\lambda^2)^{-1}$ continues in {\em odd} dimension $d$.  Thus this special case of Aharonov-Bohm scattering shares some features with odd-dimensional Euclidean scattering.


\subsection{Series expansion of the resolvent}

In this section we show that near $0$, $\tR(\lambda)$ has an expansion in powers of $\lambda^2$, $\lambda^{2|\beta|}$, and $\lambda^{2(1-|\beta|)}$.  
In fact, we shall have several operators and functions with this sort of expansion. If near $\lambda=0$, $H(\lambda)$ can be written
\begin{equation}\label{eq:pexp}
H(\lambda)=\sum_{j,k\geq 0} T_{jk}(H)\lambda^{2(j+k|\beta|)}+ \sum_{k>0, j\geq 0} T_{jk}'(H)\lambda^{2(j+k(1-|\beta|))}
\end{equation}
for some $T_{jk}(H)$, $T_{jk}'(H)$, then we shall say $H$ has an expansion in powers of $\lambda^2$, $\lambda^{2|\beta|}$, and $\lambda^{2(1-|\beta|)}$.  We shall sometimes just say $H$
has an expansion of the form \eqref{eq:pexp}.  We shall mean by this that there is an $\epsilon>0$ so
that each series converges uniformly for $|\lambda|<\epsilon$.  If $H$ is operator-valued $L^2_c\rightarrow \mathcal{D}_{\loc}$, we mean that for each $\chi\in C_c^\infty$ there is a neighborhood of $0$ (which may depend on $\chi$) such that the series converge uniformly after multiplication on both the left and the right by $\chi$.

Note that if $h_1(\lambda)$, $h_2(\lambda)$ are functions which have expansions of the form \eqref{eq:pexp}, so is $h_1h_2$.  The same is true of operator-valued $H_1(\lambda)$, $H_2(\lambda)$ 
as long as they satisfy appropriate mapping properties. The main result of this section is the following

\begin{thm}\label{thm:rnonint} There are operators $B_{jk}$, $B'_{jk}:L^2_c(\Real^2)\rightarrow \tilde{\mathcal{D}}_{\loc}$ such that near $0$
$$\tR(\lambda)= \sum_{j,k\geq 0} B_{jk}\lambda^{2(j+k|\beta|)}+ \sum_{k>0, j\geq 0} B_{jk}'\lambda^{2(j+k(1-|\beta|))}.$$
Moreover, if $k>0$ then $B_{jk}$, $B_{jk}'$ have finite rank.
\end{thm}

\begin{rem}
    That is, $\tR(\lambda)$ has an expansion of the form \eqref{eq:pexp}. Since $\tR(\lambda)$ is unitarily equivalent to $R(\lambda)$, Theorem \ref{thm:rnonint} implies the expansion of $R(\lambda)$ in Theorem \ref{t:resnonint}.  In comparing the two, recall that in
    this section we have assumed that $-1/2\leq \beta \leq 1/2,$ and that if $\beta-\beta' \in \Integers,$ then 
    $R_\beta$ and $R_{\beta'}$ are unitarily equivalent.
\end{rem}  

Theorem \ref{thm:rnonint} follows from combining Lemmas \ref{l:rs1} and \ref{l:rbded} and the observation that Lemma \ref{l:modelr}
yields that the operator $F$ defined in \eqref{eq:F} has an expansion of the 
form \eqref{eq:pexp}, with $T_{jk}(F)$, $T_{jk}'(F)$ having finite rank if $k>0$. We devote the rest of this 
paper to proving Theorem \ref{thm:rnonint}. {The following lemma is used to establish an expansion for $\tR(\lambda)$.}

\begin{lem} \label{l:rs1}
There is a finite rank operator $F_1^\sharp$ independent of $\lambda$ and an operator $D(\lambda):L^2_c(\Real^2)\rightarrow L^2_c(\Real^2)$ such that $D(\lambda)$ has an expansion as in 
\eqref{eq:pexp} and 
$$\tR(\lambda)(I-F_1^\sharp D(\lambda))=F(\lambda)D(\lambda)$$
where $F(\lambda)$ is defined by \eqref{eq:F}.
With the notation of \eqref{eq:pexp}, $T_{jk}(D)$ and $T_{jk}'(D)$  are of finite rank if $k>0$.  Moreover, if for each $\chi \in C_c^\infty(\mathbb R^2) $ and $\tilde \chi \in C_c^\infty(\Omega)$, $\|\tilde \chi \tR(\lambda)\chi\|_{L^2\rightarrow H^1}$ is bounded in some neighborhood of the origin, then 
$F_1^\sharp$ can be chosen to be $0$.
\end{lem}
\begin{proof}
We use Vodev's identity, \eqref{e:vodevid}.
Choose $\chi_2\in C_c^\infty(\Real^2)$ such that $\chi_2=1$ in a neighborhood of the support of $\chi_1$.  Using that $K(\lambda)\chi_2$ is 
a compact operator  we write
\begin{equation}\label{eq:Kdecomp}
K(\lambda)\chi_2=F_1^\sharp +K^\sharp(\lambda),\end{equation}
where $F_1^\sharp$ is finite rank and is independent of $\lambda$,  and $\|K^\sharp(0)\|_{L^2\rightarrow L^2}\leq 1/2$.   Moreover, using \eqref{eq:F} and Lemma \ref{l:modelr} we see that 
$$K^\sharp(\lambda)=K_2(\lambda)+ \lambda^{2|\beta|}K_3(\lambda)+\lambda^{2(1-|\beta|)}K_4(\lambda)$$
where $K_2,\;K_3,\;K_4$ are compact operators depending analytically on $\lambda^2$ near $0$.   Moreover, $K_3$, $K_4$ and their derivatives of all orders are of finite rank at $\lambda =0$.
In addition,  the ranges of the operators $K_2,\;K_3,\; K_4$ are contained in 
the functions with support in the support of $\chi_1$.  Then since $\|K^\sharp(0)\|\leq 1/2$, and using the support properties of $\chi_2$, there is a neighborhood of $0$ in $\Lambda$ such that
\begin{equation}\label{eq:Dexp}
D(\lambda):= (I-K^\sharp(\lambda) -K(\lambda)(1-\chi_2))^{-1}
=\sum_{m=0}^\infty  (K^\sharp(\lambda))^m(I+K(\lambda)(1-\chi_2))\end{equation}
is a bounded operator.  Moreover, because $|\beta|+(1-|\beta|)=1$, it has a series expansion of the form 
$$D(\lambda)= \sum_{j,k\geq 0} T_{jk}(D)\lambda^{2(j+k|\beta|)}+\sum _{j\geq 0,k>0} T'_{jk}(D) \lambda^{2(j+k(1-|\beta|))}$$
for $\lambda$ sufficiently near $0$.  The operators $T_{jk}(D)$ and $T_{jk}'(D)$ map $L^2_c(\Real^2)$ to $\mathcal{D}_c $, and for $k > 0$ they have finite rank.

To complete the proof, we show that if $\|\tilde \chi \tR(\lambda)\chi_2\|_{L^2\rightarrow H^1}$ is bounded in some neighborhood of $0$, then we can take $F^\sharp=0$. 
 Under these assumptions,
 $\|K_1 \chi_1\|= O(1)$ as $z\rightarrow 0$ with $z$ in the upper half plane.  
 Then, using $\|\chi_2(R_{\beta}(\lambda)-R_{\beta}(0))\chi_2\|=O(\lambda^{2|\beta|})$, we have
 \begin{align*}
 K(\lambda)\chi_2& = (\lambda^2 - z^2)\chi_1(2-\chi_1)\tR (z) \chi_2 + [P_{\beta}, \chi_1](R_{\beta}(\lambda) - R_{\beta}(z))K_1\chi_2\\
 & = O(\lambda^{2|\beta|}) + O(z^{2|\beta|}).
\end{align*}
Thus we can choose $z=i|z|$ sufficiently small in norm and then $\lambda_0$ sufficiently small so that 
$\| K(\lambda)\chi_2\|\leq 1/2$ when $|\lambda|<\lambda_0$.  This allows us to choose $F_1^\sharp=0$ in \eqref{eq:Kdecomp}, since $\| K(0)\chi_2\|<1/2$.
\end{proof}

To show Theorem \ref{thm:rnonint} using the previous lemma, we only need the following
\begin{lem}\label{l:rbded}
For any $\chi \in C_c^\infty(\Real^2)$, both $\| \chi R(\lambda)\chi\|_{L^2\rightarrow \mathcal{D}}$ and $\| \chi \tR(\lambda)\chi\|_{L^2\rightarrow \tilde{\mathcal{D}}}$ are uniformly bounded in a neighborhood of $\lambda = 0$.
\end{lem}
In fact, in order to prove Theorem \ref{thm:rnonint} using Lemma \ref{l:rs1}, it is enough to show that  $\| \tilde \chi R(\lambda)\chi\|_{L^2\rightarrow H^1}$ is bounded when $\tilde \chi \in C_c^\infty(\Omega)$. Such a resolvent bound follows directly from Lemma \ref{l:rbded}, as $\mathcal{D}\subset H^1$ by the definition of Friedrichs extension, or even from basic local elliptic regularity since $\tilde\chi$ is supported away from $\Gamma$. Also see \cite[Corollary 2.5]{Le}. We postpone the proof of Lemma \ref{l:rbded} to the next subsection.  Its proof uses the first part of Lemma \ref{l:rs1}.

We also need the following result on the null space of the conjugated operator $\tilde P$.
\begin{lem}\label{l:nores}
If $\tilde P u=0$, $u\in \tilde{D}_{\loc}$ and $u$ is bounded, then $u$ is identically zero. 
\end{lem}

\begin{proof}
    Let $D_\rho = \{x \in \mathbb R^2 \colon |x|<\rho\}\supset \Gamma$. In $\Real^2\setminus D_{\rho}$,
     we can write $\tilde P$  as 
    $$\tilde P = D^2_r - \frac{i}{r}D_r + \frac{1}{r^2}({D_{\theta}-\beta})^2, \quad \text{for } r>\rho.$$
    Using this, separating variables, and 
    using asymptotics of Bessel functions, the bounded solutions of $\tilde P u = 0$ are given by
    \begin{equation}
    \label{e:aux1}
        u = \sum_{k\in\mathbb{Z}} c_k r^{-{|k-\beta|}} e^{ik\theta} , \quad \text{for } r>\rho.
    \end{equation}
    Define the magnetic normal derivative as $\partial_{\nu}^{\Vec{A}}u := \nu \cdot (\nabla - i\Vec{A})u$, where $\nu$ is the outward unit normal vector at the boundary. We then have the magnetic Green's identity:
    \begin{equation}
    \label{e:mag_green}
        \int_{\Omega} (-i\nabla-\beta\Vec{A_0}) u \cdot \overline{(-i\nabla-\beta\Vec{A_0})v} = \int_\Omega (\tilde{P} u) \bar v + \int_{\partial\Omega} (\partial_{\nu}^{\beta\Vec{A}_0} u ) \bar v
    \end{equation}    
    where as before we distinguish, for $z=(x,y)\in \Gamma$, $\lim_{\epsilon \rightarrow 0^{+}} e^{\pm i \epsilon} z$.
    Then as in the proof of Lemma~\ref{l:quadform}, for $u,v\in \tilde{\mathcal{D}}$ by the matching condition \eqref{eq:matching} and magnetic Green's identity, we have
    \begin{equation}
    \label{e:aux2}
         \int_{\Omega} (-i\nabla-\beta\Vec{A_0}) u \cdot \overline{(-i\nabla-\beta\Vec{A_0}) v}  = \int_\Omega (\tilde{P} u) \bar v.
    \end{equation}
    By \eqref{e:aux1} and \eqref{e:aux2}, for solutions of $\tilde P u =0$, we have 
    \[
    \int_{\Omega \cap D_\rho} |(-i\nabla-\beta\Vec{A_0}) u|^2 = \int_{\partial D_\rho} \bar u\, (\partial_{\nu}^{\beta\Vec{A}_0} u) \, dS =  \int_{\partial D_\rho} \bar u\, (\partial_{r} u) \, dS= O(\rho^{-2|\beta|}), \qquad \text{ as } \rho \to \infty,
    \]
    which follows from the fact that the normal vector of $\partial D_\rho$ is perpendicular to $\Vec{A}_0$.
    This implies that $(i\nabla + \beta\Vec{A_0}) u$ is identically $0$ on $\Omega$. As $i\nabla + \beta\Vec{A_0}$ is unitarily equivalent to $i\nabla$ on $\Omega$ with an additional cut from $s_1$ to $\infty$ when $\beta\notin\mathbb Z$, using $u \to 0$ at points of $S$, we conclude that $u$ is identically zero.
\end{proof}

\subsection{Proof of Lemma \ref{l:rbded}}
Now we prove Lemma \ref{l:rbded}, which shows that for any $\chi \in C_c^\infty(\Real^2)$, $\| \chi R(\lambda)\chi\|_{L^2\rightarrow \mathcal{D}} $ is bounded at $0$. 

\begin{lem}\label{l:bexp}
There are operators $S_{jk},\; S'_{jk}: L^2_c(\Real^2)\rightarrow \Tilde{\mathcal{D}}_{\loc}$ and $\mu_0\geq 0$ such that near $0$
\begin{equation}\label{eq:basicR}
\tR(\lambda)= \lambda^{-\mu_0} \left( \sum_{j,k\geq 0} S_{j,k} \lambda^{2(j+k|\beta|)}+ \sum_{j\geq 0, k>0} S'_{j,k}\lambda^{2(j+k(1-|\beta|))}\right).
\end{equation}
\end{lem}
\begin{proof}
 Our argument uses a variation on the proof of the standard analytic Fredholm Theorem; see, e.g.  \cite[Theorem VI.14]{RS} for a 
proof of the classical result, and 
\cite[Theorem~4.1]{MuSt} for a more abstract setting than ours.

From Lemma \ref{l:rs1}
it suffices to invert $I-F_1^\sharp D(\lambda)$ near $\lambda =0$.  By the usual analytic Fredholm theorem, we know that this must be invertible for all 
but isolated values of $\lambda \in \Lambda\setminus\{0\}$, since $I-K(\lambda)$ is invertible for $\lambda=z$.  Since $I-F^\sharp D(\lambda)$
 differs from the identity by a finite rank operator, it  can be inverted essentially by using Cramer's rule.   In doing so, we get a denominator  coming from a determinant of a matrix.  This 
 determinant $d(\lambda)$ has an expansion at $0$ of the form  $d(\lambda)=\sum_{j,k\geq 0} d_{jk}\lambda^{2(j+k|\beta|)}+\sum _{j\geq 0,k>0} d'_{jk} \lambda^{2(j+k(1-|\beta|))}$.  Since this 
 determinant is not identically zero,
 there is a unique $\mu_0\geq 0$ so that $\lim_{\lambda \rightarrow 0} d(\lambda)\lambda^{-\mu_0}$ is finite and nonzero.  We call this  limit  $d_0$, with $d_0 \ne 0$. Then 
 $$\frac 1 {d(\lambda)}=\frac 1 {\lambda^{\mu_0}d_0} \frac 1 {(1+d_0^{-1}(\lambda^{-\mu_0}d(\lambda)-d_0))}=\frac 1 {\lambda^{\mu_0}}\bigg( \sum_{j,k\geq 0} f_{jk}\lambda^{2(j+k|\beta|)}+\sum _{j\geq 0,k>0} f'_{jk} \lambda^{2(j+k(1-|\beta|))}\bigg)$$
 for some constants $f_{jk},\; f_{jk}'.$
 Using this and \eqref{eq:Dexp}, $I-K(\lambda)$ has an inverse $D(\lambda)(I-F^\sharp_1 D(\lambda))^{-1}$ which has an expansion of the form \eqref{eq:pexp}.  
 Since by \eqref{e:vodevid}, $\tR(\lambda)=F(\lambda)(I-K(\lambda))^{-1}$, and \eqref{eq:F} and Lemma \ref{l:modelr} give an expansion in $\lambda$ of $F$, this proves the lemma.
 \end{proof}

The $\mu_0$ in the previous lemma gives an upper bound on what might, in analogy with meromorphic functions, be called the order of 
the singularity at $0$ of $\tR(\lambda)$.  It is possible, however, that $\tR(\lambda)$ is not that singular at $0$.    Define
$\mu_1\geq 0$ by 
\begin{equation}\label{eq:mu1}
\mu_1=\inf\{ \mu\in \Real: \; \lim_{\lambda\rightarrow 0} \| \lambda^\mu \chi \tR(\lambda)\chi \| <\infty\; \text{for any $\chi \in C_c^\infty(\Real^2)$}\} .
\end{equation}
Note that $\mu_1\leq \mu_0$ and that by Lemma \ref{l:bexp} $\lim_{\lambda\rightarrow 0} \lambda^{\mu_1}\tR(\lambda)\not =0$.
\begin{lem}\label{l:singterm}
If $\mu_1>0$ and $U$ is in the range of $\lim_{\lambda\rightarrow 0} (\lambda^{\mu_1} \tR(\lambda))$, then $U\in \Tilde{\mathcal{D}}_{\loc}$, $\Tilde{P} U=0$ and $U(x)=o(1)$ as $|x|\rightarrow \infty$. 
\end{lem}
\begin{proof}  That $U\in \Tilde{\mathcal{D}}_{\loc}$ follows from Lemma \ref{l:bexp}.   Using Lemma \ref{l:bexp} shows that $\tR(\lambda)$ has an expansion at $0$ and the identity 
$(\tilde P-\lambda^2)\tR(\lambda)=I$ yields from the coefficient of $\lambda^{-\mu_1}$ that $\tilde P U=0$.

Set $S_{-\mu_1}=\lim_{\lambda\rightarrow 0} (\lambda^{\mu_1} \tR(\lambda))$.    We shall  use 
\eqref{e:vodevido}, a variant of Vodev's identity.
  This time, though, we shall fix $\lambda$ with $\im \lambda>0$ and expand $R_\beta(z)$, $\tR(z)$ in $z$ near $z=0$.  Then from the coefficient of $z^{-\mu_1}$ in
\eqref{e:vodevido} 
we get
$$-S_{-\mu_1}=\lambda^2\tR(\lambda) \chi_1 (2-\chi_1)S_{-\mu_1} - \{ 1-\chi_1+\tilde R(\lambda)[P_\beta,\chi_1]\}(R_\beta(\lambda)-R_{00})[P_\beta,\chi_1]S_{-\mu_1}.$$
Using that $(\tR(\lambda)\phi)(x),\;(R_\beta(\lambda)\phi)(x)=O(e^{-|x|\im \lambda})$ as $|x|\rightarrow \infty$ for any $\phi\in L^2_c(\Real^2)$ yields
$$(S_{-\mu_1}\phi)(x)= (R_{00}[P_\beta,\chi_1]S_{-\mu_1}\phi)(x)+ O(e^{-|x|\im \lambda}).$$
Observing that Lemma \ref{l:modatinfty} implies that  $(R_{00}\psi)(x)$ is decaying as $|x|\rightarrow \infty$ for any $\psi \in L^2_c(\Real^2)$ completes the proof. 
\end{proof}

These two lemmas allow us the prove the main result of this subsection.

\begin{proof}[Proof of Lemma \ref{l:rbded}]
Suppose  that
 $\lim\sup_{\lambda \rightarrow 0}\| \chi \tR(\lambda)\chi\|_{L^2\rightarrow L^2}=\infty$ for some $\chi\in C_c^\infty(\Real^2)$.  Then {Lemma~\ref{l:bexp} implies that} the $\mu_1$ defined in \eqref{eq:mu1} satisfies $\mu_1>0$, and by Lemma 
\ref{l:singterm} there is a $U\in \tilde{\mathcal{D}}_{\loc}$ with $\tilde P U=0$ and $U(x)=o(1)$ as $|x|\rightarrow \infty$.  But then this contradicts Lemma \ref{l:nores}, so that $\lim\sup_{\lambda \rightarrow 0}\| \chi \tR(\lambda)\chi\|_{L^2\rightarrow L^2}$ is finite.  By Lemma \ref{l:bexp} the limit must exist.

To show that $\lim_{\lambda \rightarrow 0}\| \chi \tR(\lambda)\chi\|_{L^2\rightarrow \tilde{\mathcal{D}}}$ is finite, note that $\tilde{P} \tR(\lambda)=I+\lambda^2\tR(\lambda)$. Thus 
\begin{equation}
    \|\tilde{P}\chi \tR(\lambda)\chi\|_{L^2\to L^2} \leq \|[\tilde{P},\chi] \tR(\lambda)\chi\|_{L^2\to L^2} + \|\chi (I+\lambda^2\tR(\lambda))\chi\|_{L^2\to L^2}<\infty.
\end{equation}
Lemma \ref{l:rbded} then follows from taking the unitary conjugation. 
\end{proof}


\begin{thebibliography}{0}


\bibitem[AhBo59]{ab} Y. Aharonov and D. Bohm. \textit{Significance of electromagnetic potentials in the quantum theory}. Physical Review 115.3 (1959): 485.

\bibitem[AdTe98]{at} R. Adami, and A. Teta. \textit{On the Aharonov–Bohm Hamiltonian}. Letters in Mathematical Physics 43 (1998): 43-54.



\bibitem[BGD88]{BGD} D. Boll\'{e}, F. Gesztesy, C. Danneels, \textit{Threshold scattering in two dimensions.} Ann. Inst. H. Poincaré Phys. Théor. 48:2 (1988), 175--204.


\bibitem[ChDa23a]{cd} T. J. Christiansen and K. Datchev. \textit{Low energy scattering asymptotics for planar obstacles}. Pure and Applied Analysis 5.3 (2023): 767--794.

\bibitem[ChDa25]{cd2} T. J. Christiansen and K. Datchev. \textit{Low energy resolvent expansions in dimension two}. Communications of the American Mathematical Society 5.02 (2025): 48-80.

\bibitem[CDY25]{cdy} T. J. Christiansen, K. Datchev, and M. Yang.
\textit{From resolvent expansions at zero to long time wave expansions}. Communications in Partial Differential Equations 50.4 (2025): 477-492.

\bibitem[CoFe25]{cf23} M. Correggi, and D. Fermi. \textit{Schr\"{o}dinger operators with multiple Aharonov-Bohm fluxes.} Annales Henri Poincaré. 26.1 (2025): 123-163.

\bibitem[DaRo13]{dr} M. Dafermos and I. Rodnianski, \textit{Lectures on black holes and linear waves.} Clay Mathematics Proceedings 17 (2013) 105--205.



\bibitem[Da\v{S}\v{t}98]{ds} L. Dabrowski, and P. {\v{S}}\v{t}ov{\'i}{\v{c}}ek. 
\textit{Aharonov--Bohm effect with $\delta$-type interaction}. Journal of Mathematical Physics 39.1 (1998): 47-62.

\bibitem[DyZw19]{dz} S. Dyatlov and M. Zworski,  \textit{Mathematical theory of scattering resonances.} Grad. Stud. Math. 200. Amer. Math. Soc., 2019.

\bibitem[FFFP13]{fffp} L. Fanelli, V. Felli,  M. A. Fontelos, and A. Primo. 
\textit{Time decay of scaling critical electromagnetic Schrödinger flows.} Communications in Mathematical Physics 324 (2013): 1033-1067.

\bibitem[FZZ22]{fzz} L. Fanelli, J. Zhang, and J. Zheng. \textit{Dispersive estimates for 2D-wave equations with critical potentials.} Advances in Mathematics 400 (2022): 108333.

\bibitem[Fer24]{f24} D. Fermi. \textit{The Aharonov-Bohm Hamiltonian: self-adjointness, spectral and scattering properties.} Preprint arXiv:2407.15115 (2024).



\bibitem[GrKo14]{gk} G. Grillo, and H. Kova\v{r}\'{i}k, \textit{Weighted dispersive estimates for two-dimensional Schrödinger operators with Aharonov–Bohm magnetic field.} Journal of Differential Equations 256.12 (2014): 3889-3911.

\bibitem[Jen80] {je80} A. Jensen, \textit{Spectral properties of Schr\"{o}dinger operators and time-delay of the wave functions: results in $L^2(\Real^m)$, $m\geq 5$}.  Duke Math. J. 47:1 (1980), 57--80.

\bibitem[Jen84]{je84} A. Jensen, \textit{Spectral properties of Schrödinger operators and time-decay of the wave functions. Results in $L^2(\Real^4)$.}  J. Math. Anal. Appl. 101:2 (1984), 397--422.

\bibitem[JeKa79]{JeKa} A. Jensen and T. Kato, \textit{Spectral properties of Schrödinger operators and time-decay of the wave functions.} Duke Math. J. 46:3 (1979),  583--611. 

\bibitem[JeKo23]{je23} A. Jensen, and H. Kova\v{r}\'{i}k. \textit{Resolvent expansions of 3D magnetic Schr\"odinger operators and Pauli operators.} Journal of Mathematical Physics 65.7 (2024).


\bibitem[JeNe01]{jn} A. Jensen and G. Nenciu, \textit{A unified approach to resolvent expansions at thresholds}. Rev. Math. Phys. 13:6 (2001), 727--754. Erratum 16:5 (2004), 675--677.

\bibitem[Kla23]{klainerman} S. Klainerman. \textit{Columbia lectures on the stability of Kerr}. Preprint \url{https://www.math.columbia.edu/~staff/columbia2023.pdf}, 2023.

\bibitem[Kov15]{Kov2} H. Kova\v{r}\'{i}k, \textit{Resolvent expansion and time decay of the wave functions for two-dimensional magnetic Schrödinger operators}. Communications in Mathematical Physics 337 (2015), 681-726.

\bibitem[Kov22]{Kov} H. Kova\v{r}\'{i}k, \textit{Spectral properties and time decay of the wave functions of Pauli and Dirac operators in dimension two}. Adv. Math. 398 (2022), Paper No. 108244, 94 pp.


\bibitem[LaPh89]{lp89} P. D. Lax and R. S. Phillips. \textit{Scattering Theory: Revised Edition}. Academic Press, Inc. 1989.

\bibitem[Len15]{Le} C. Léna, \textit{Eigenvalues variations for Aharonov-Bohm operators}. Journal of Mathematical Physics, 56(1), 2015.

\bibitem[Mac65]{maccamy} R. C. MacCamy. \textit{Low frequency acoustic oscillations}. Quarterly of Applied Mathematics 23:3 (1965) 247--255.

\bibitem[Mor61]{m61} C. S. Morawetz, \textit{The decay of solutions of the exterior initial-boundary value problem for the wave equation}. Communications on Pure and Applied Mathematics, 14 (1961), 561--568.

\bibitem[M\"uSt14]{MuSt}  J. M\"uller and A. Strohmaier. \textit{The theory of Hahn-meromorphic functions, a holomorphic Fredholm theorem, and its applications}. Anal. PDE 7:3 (2014), pp.745--770.

\bibitem[Mur82]{Mur} M. Murata, \textit{ Asymptotic expansions in time for solutions of Schrödinger-type equations.}
J. Funct. Anal. 49 (1982), no. 1, 10--56. 


\bibitem[Ran10]{ransford} Thomas Ransford. \textit{Computation of Logarithmic Capacity}. Computational Methods and Function Theory, 10:2, pp. 555--578, 2010.

\bibitem[ReSi80]{RS} M. Reed and B. Simon, \textit{Methods of Modern Mathematical Physics. I. Functional Analysis}. Second edition. Academic Press, Inc. [Harcourt Brace Jovanovich, Publishers], New York, 1980.


\bibitem[Sch21]{sch21} W. Schlag, \textit{On pointwise decay of waves}. J. Math. Phys. 62 (2021) 061509.


\bibitem[SjZw91]{sz} J. Sj\"ostrand and M. Zworski, \textit{Complex scaling and the distribution of scattering
poles.} J. Amer. Math. Soc. 4:4 (1991), 729--769.

\bibitem[Tat13]{tat} D. Tataru, \textit{Local decay of waves on asymptotically flat stationary space-times} Amer. J. Math. 135:2 (2013) 361--401.


\bibitem[Vai89]{va} B. Vainberg. \textit{Asymptotic Methods in Equations of Mathematical Physics}. CRC Press, 1989.

\bibitem[Vas20]{vasy} A. Vasy. \textit{The black hole stability problem.} Current Developments in Mathematics 2020 (2020), 105--155.

\bibitem[Vod14]{vodev} G. Vodev. \textit{Semi-classical resolvent estimates and regions free of resonances}. Math.
Nachr. 287:7 (2014), 825--835.

\bibitem[WZZ23]{wang23} H. Wang, F. Zhang, and J. Zhang. \textit{Decay estimates for one Aharonov-Bohm solenoid in a uniform magnetic field II: wave equation.} arXiv preprint arXiv:2309.07649 (2023).

\bibitem[Wat22]{watson} G. N. Watson. \textit{Treatise on the Theory of Bessel Functions}. Cambridge University Press, 1922.

 
\bibitem[Yan22]{my} M. Yang. \textit{Resolvent estimates for the magnetic Hamiltonian with singular vector potentials and applications}. Commun. Math. Phys. 394 (2022), 1225-1246.

\end{thebibliography}
\end{document}